\documentclass{article}
\usepackage{amsthm}
\usepackage{amsmath}
\usepackage[authoryear,round]{natbib}
\usepackage[colorlinks=true,linkcolor=blue,citecolor=blue]{hyperref}
\usepackage{hypernat}
\usepackage{amssymb}
\usepackage{verbatim}
\usepackage{xr}
\usepackage{graphicx}
\usepackage{stackrel}

\usepackage{array}
\newcolumntype{H}{>{\setbox0=\hbox\bgroup}c<{\egroup}@{}}

\newcommand{\R}{{\mathbb R}}
\newcommand{\E}{{\mathbb E}}
\renewcommand{\P}{{\mathbb P}}
\newcommand{\N}{{\mathbb N}}

\newcommand{\eps}{{\varepsilon}}

\DeclareMathOperator{\Var}{Var}

\DeclareMathOperator{\trace}{trace}

\newtheorem{theorem}{Theorem}[section]

\newtheorem{lemma}[theorem]{Lemma}

\newtheorem{remark}[theorem]{Remark}
\newtheorem*{remark*}{Remark}

\newtheorem*{definition*}{Definition}

\newcounter{rcnt}[section]

\newcounter{desccount}

\newcommand{\descref}[1]{\hyperref[#1]{#1}}

\begin{document}
\sloppy

\title{Statistical inference with $F$-statistics
        when fitting simple models to high-dimensional data}

\author{
	Hannes Leeb (University of Vienna and DataScience@UniVienna)
	\\
	Lukas Steinberger (University of Freiburg)
	}

\maketitle

\begin{abstract}
We study linear subset regression in the context of the
high-dimensional overall model $y = \vartheta+\theta' z + \epsilon$
with univariate response
$y$ and a $d$-vector of random regressors $z$, independent of
$\epsilon$. Here,
‘high-dimensional’ means that the number $d$ of available explanatory
variables is much larger than the number $n$ of observations.
We consider simple linear sub-models where $y$ is regressed on a set
of $p$ regressors given by $x = M'z$, for some $d \times p$ matrix
$M$ of full rank $p < n$. The corresponding simple model, i.e.,
$y=\alpha+\beta' x + e$, can be
justified by imposing appropriate restrictions on the unknown
parameter $\theta$ in the overall model; otherwise,
this simple model can be
grossly misspecified. In this paper,
we establish asymptotic validity of the standard 
$F$-test on the surrogate
parameter $\beta$, in an appropriate sense,
even when the simple model is misspecified.
\end{abstract}

\section{Introduction}
\label{introduction}

The $F$-test is a staple tool of applied statistical analyses.
It is widely used, sometimes also in situations where its
applicability is debatable because underlying assumptions may not be met.
We study a situation of this kind: An $F$-test after fitting a
(possibly misspecified) working model.
We focus, in particular, on a scenario where
the fitted model has $p$ explanatory variables while
the true model has $d$ explanatory variables,
with $p \ll d$, and where sample
size is of the same order as $p$, i.e., $p = O(n)$.
Scenarios like this occur, for example, 
in quality control
studies like~\cite{Sou91a}, where a model with 18 explanatory variables
(out of a total of about 8,000) is fit based on a sample of size 50;
in time series forecasting with principal components as in \cite{Sto02a},
who extract a handful of factors from 149 explanatory variables based on
480 monthly observations;
or
in genetic analyses like~\cite{Vee02a}, who select and fit a model with 
70 genes (out of a total of about 25,000) based on a sample of size 78.
In situations like these, the question whether the fitted model has 
any explanatory value is of particular interest. We show that, approximately, 
the usual $F$-statistic is 
$F$-distributed under a corresponding null-hypothesis,
and that it is non-central $F$-distributed in a local neighborhood
of the null. Approximation errors go to zero as $n\to \infty$ if
$n^2 / \log d \to 0$ and if, at the same time, 
$p$ is of the same, or of slower, order as $n$; 
cf. Theorem~\ref{t1} and Remark~\ref{rateofp}, respectively.
Our results are uniform over a large region of the parameter
space that we consider. In particular, our results also cover situations
where the fitted model is misspecified.
The setting of our analysis is non-standard in that we 
require a particular constellation of $d$, $p$ and $n$.
This is a challenging setting of 
practical relevance, for which few theoretical results are
available so far.
Our findings, which are given for independent 
observations, also prompt the question whether similar 
results can be obtained under serial correlation.

The $F$-statistic is exactly $F$-distributed in a correctly specified
linear model with Gaussian errors; and it is asymptotically $F$-distributed
under the strong Gau{\ss}-Markov condition on the errors if $n\to\infty$ while
the model dimension stays fixed; cf.~\citet{And58a}.
$F$-tests  in correctly specified models in settings where $p$ is allowed 
to increase
with $n$ are studied, among others, by
\cite{Akr00a, Bat05a, Boo95a, Har08a, Por84a, Por85a,Wan13a}.
In addition, there are several viable alternatives to the $F$-test in
potentially misspecified settings;
see, for example, \citet{Che10b, Eic67a, Hub67a, Whi80a, Whi80b, Zho11a}.
For further results on hypothesis testing and marginal screening in 
misspecified models, see, for example,
\citet{Boo13a, Cho11a, Fom03a, Jen91a, Ram91a}, and the references therein.

On a technical level, this paper relies on~\citet{Wan13a}, the corresponding 
extensions and corrections in~\cite{Ste16a}, 
and also on~\citet{Ste18a, Ste18b};
all but the first of these references
are based on \citet{Ste15}.

The rest of the paper is structured as follows: 
In Section~\ref{thetruemodel}, we describe the true data-generating
model and the underlying parameter space. 
The (typically misspecified) working model and
the corresponding $F$-statistic are described in Section~\ref{theworkingmodel}.
Our main theoretical result is given in Section~\ref{mainresult},
and a simulation study in Section~\ref{numericalresults} demonstrates
that our asymptotic approximations can `kick-in' reasonably fast.

\section{The true model}
\label{thetruemodel}

Throughout, we consider the (true) linear model
\begin{equation}\label{y}
y \quad = \quad  \vartheta+\theta' z + \epsilon
\end{equation}
with $\vartheta\in \R$ and $\theta \in \R^d$ for some $d \in \N$.
We assume that the error $\epsilon$ is independent of $z$, with
mean zero and finite variance $\sigma^2>0$; its distribution will be denoted by 
$\mathcal L(\epsilon)$.
Moreover, we assume that the vector of regressors  $z$ has mean $\mu \in \R^d$
and positive definite variance/covariance matrix $\Sigma$.
Our model assumptions are further discussed in 
\citet[Remark 7.1]{Ste18a}.
No additional restrictions will be placed on
the regression coefficients $\vartheta$ and $\theta$, on the
moments $\mu$ and $\Sigma$, 
or on the error distribution $\mathcal L(\epsilon)$.

We do place some assumptions on the distribution of the explanatory 
variables. First, we assume  that $z$ can be written as
an affine transformation of independent random variables. With this, 
we can represent the $d$-vector $z$ as 
\begin{equation}\label{z}
z \quad=\quad \mu + \Sigma^{1/2} R \tilde{z}
\end{equation}
for a $d$-vector $\tilde{z}$ with independent (but not necessarily
identically distributed) components so that
$\E[\tilde{z}]=0$ and $\E[\tilde{z}\tilde{z}'] = I_d$,
where $\Sigma^{1/2}$ is the positive definite and symmetric square root 
of $\Sigma$, and where $R$ is an orthogonal (non-random) matrix.
Second, we assume that $\tilde{z}$ has a Lebesgue density, which we denote 
by
$f_{\tilde{z}}$, with bounded marginal densities and finite marginal 
moments of sufficiently high order. In particular, we will assume that
$f_{\tilde{z}}$ belongs to one of the classes ${\mathcal F}_{d,k}(D,E)$
that are defined in the next paragraph, for appropriate constants 
$k$, $D$ and $E$. Our assumptions on $z$ are similar to those
maintained by~\citet{Bai96a} and~\cite{Zho11a}.
For later use, note that the distribution of  $(y,z)$ in \eqref{y}--\eqref{z}
is characterized by $\vartheta$ and $\theta$,
by $\mathcal L(\epsilon)$, by $\Sigma$ and $\mu$, by $f_{\tilde{z}}$, and
by $R$.  

Fix an integer $k\geq 1$ and positive (finite) constants $D$ and $E$.
With this, write ${\mathcal F}_{d,k}(D,E)$ for the class of Lebesgue densities
on $\R^d$ that are products of univariate marginal densities such that
each such marginal density is bounded from above by $D$, and such that
each univariate marginal density has absolute moments of order up to $k$
that are bounded by $E$.

\section{The sub-model and the $F$-test}
\label{theworkingmodel}

Consider a sub-model where $y$ is regressed on $x$, with $x$ given by
\begin{equation} \label{x}
x \quad=\quad  M'z
\end{equation}
for some full-rank $d\times p$ matrix $M$ with $p<d$.
For example, $M$ can be a selection matrix that picks out $p$ components
of the $d$-vector $z$. 
Submodels with regressors of the form $x=M'z$ also occur in
principal component regression,
partial least squares, and certain sufficient dimension reduction
methods.
We are particularly interested in situations where $d$ 
is \emph{much} larger than $p$, i.e., $p \ll d$.
Trivially, we can write 
\begin{equation} \label{workingmodel}
	y = \alpha + \beta' x + e
\end{equation}
with $e = y - \alpha-\beta'x$, where $\alpha$ and $\beta$
minimize $\E[(y-\alpha-\beta' x)^2]$. 
The `error' $e$ has mean zero (because both \eqref{y} and
\eqref{workingmodel} include an intercept), and we denote
its variance by $s^2 = \E[e^2]$.
Note that
$\alpha= \vartheta+\mu'\theta - \mu' M (M' \Sigma M)^{-1} M' \Sigma \theta$
and, for later use, that
\begin{align} \label{betas2}
\begin{split}
\beta \quad & = \quad (M'\Sigma M)^{-1} M'\Sigma \theta\quad\text{and}\\
s^2 \quad   & = \quad \theta'\Sigma\theta +
		\theta'\Sigma M (M'\Sigma M)^{-1} M'\Sigma\theta+\sigma^2.
\end{split}
\end{align}
Irrespective of whether the working model is correctly specified,
the `surrogate' parameters $\alpha$, $\beta$ and $s^2$ are always 
well-defined. Here, $\beta$ is our main object
of interest, instead of the underlying true parameter $\theta$.
Such surrogate parameters are well-known in the statistics literature, 
certainly since \citet{Hub67a}, and have recently gained new popularity, as 
witnessed by, e.g., \citet{Aba14a, Bra14a, Bac15a, Buj14a}. In 
particular, such surrogate parameters can be consistently estimated, 
in a standard $M$-estimation setting, by the OLS estimator or by robust 
alternatives, provided that $p$ is not too large relative to $n$ 
\citep[see][]{Por84a, Por85a, Whi80a, Whi80b}; 
cf. also Lemma~A.3 in~\citet{Ste15} and Lemma~A.4 in~\citet{Ste18a}
for analyses tailored to our present setting.

The working model \eqref{workingmodel} is correct (in the usual sense)
if $\E[y\|z] = \E[y\|x]$, i.e., if
$\vartheta+\theta'z = \alpha + \beta'x$ or, equivalently, if $\epsilon=e$.
This is the case if $\theta$ lies in the column space of $M$;
if $M$ is a selection matrix, this means that $M'\theta$ selects all the 
non-zero components of $\theta$.
Here, we do not assume that the working model is correct.
In particular, we stress that
$e$ may differ from $\epsilon$,  and that $e$ may depend on $x$.

When working with the simple sub-model \eqref{workingmodel}, 
a natural question is whether $x$ has any explanatory value for the response 
variable $y$. 
Given a sample of $n > p+1$ independent and identically distributed (i.i.d.) 
observations of $y$ and $x$ 
from \eqref{workingmodel},
a classical approach to this question is to use the $F$-test
of the hypotheses
\begin{equation}\label{hypotheses}
H_0: \beta = 0 \quad\text{versus}\quad H_1: \beta\neq 0.
\end{equation}
Let $Y$ and $X$ denote the $n\times 1$ vector of responses 
and the $n\times p$ matrix of explanatory variables, respectively.
Write $\hat{\beta}$ for the OLS-estimator
for $\beta$ when $Y$ is regressed on $X$ and a constant, set
$\hat{s}^2 = \|(I_n - P_{\iota,X})Y\|^2/(n-p-1)$, and
write $\hat{F}_n=\hat{F}_n(X,Y)$ for the 
usual $F$-statistics for testing $H_0$, i.e.,
$\hat{F}_n= \|(I_n - P_\iota)X \hat{\beta}\|^2/ (p \hat{s}^2)$
if the numerator is well-defined and the denominator is positive
and $\hat{F}_n=0$ otherwise. Here, $P_{\dots}$ denotes the orthogonal
projection on the space spanned by the column-vectors 
indicated in the subscript
and $\iota$ denotes the $n$-vector $\iota=(1,\dots,1)'$.
Note that $\hat{F}_n>0$ with probability one
by our assumptions.

$H_0$ may be re-phrased as the hypothesis
that the best linear predictor of $y$ given $x$ is constant.
An alternative to $H_0$ is the hypothesis
that the Bayes-estimator of $y$ given $x$ is constant, i.e.,
$$
\tilde{H}_0: \E[y\|x] \text{ is constant.}
$$
Testing this non-parametric hypothesis is more difficult.
In the asymptotic setting that we consider in the next section, however,
we find that $H_0$ and $\tilde{H}_0$ are close to each other in the 
sense that the Bayes predictor and the best linear predictor (of $y$ given $x$)
are close in terms of mean-squared prediction error; see
Remark~\ref{pinsker} for details.

\section{Main result}
\label{mainresult}

Our main result is concerned with the asymptotic distribution of 
the $F$-statistic in a local neighborhood of the null-hypothesis.
Here, the local neighborhood
is defined through the requirement that 
$$
	\Delta \quad=\quad 
	\text{Var}(  \beta' x) / \text{Var}( e)
	\quad=\quad
	\beta' M' \Sigma M \beta /s^2
$$
is small.
This quantity can be interpreted as a signal-to-noise ratio
in \eqref{workingmodel} and 
depends on $\theta$, $M$, $\Sigma$ and $\sigma^2=\E[\epsilon^2]$; 
cf. \eqref{betas2}.
If the error $e$ in \eqref{workingmodel} is Gaussian and independent of $x$, then
the $F$-statistic $\hat{F}_n$ is
$F$-distributed with parameters $p$, $n-p-1$ and non-centrality
parameter $n \Delta$; in that case, we have $\P(\hat{F}_n \leq t) =
F_{n,n-p-1,n\Delta}(t)$, where $F_{n,n-p-1,n \Delta}(\cdot)$ denotes
the cumulative distribution function (c.d.f.) 
of the $F$-distribution with indicated parameters.
In our present setting, however, the error $e$ in \eqref{workingmodel}
need not be Gaussian and can depend on $x$.

We will show that the distribution of $\hat{F}_n$ can be approximated
by an $F$-distribution, uniformly over most parameters in the model.
Only for $\epsilon$, $f_{\tilde{z}}$ and $R$, i.e., 
for the error in \eqref{y} and
for the density of the standardized explanatory variables 
as well as the orthogonal matrix in \eqref{z}, some restrictions are needed.
We will require a moment restriction on $\epsilon/\sigma$, and
we will require that $f_{\tilde{z}}$ belongs to one of the classes
${\mathcal F}_{d,k}(D,E)$ introduced earlier.
To formulate the restriction on $R$,
write ${\mathcal O}_{d}$ for the
collection of all orthogonal $d\times d$ matrices and
write $\nu_{d}$ for the uniform distribution on that set; i.e.,
$\nu_{d}$ is the normalized Haar measure on the $d$-dimensional
orthogonal group. 
For $R$, we will require that it  belongs 
to a Borel set ${\mathbb U}\subseteq {\mathcal O}_{d}$
that is large in terms of $\nu_{d}$.

\begin{theorem}\label{t1}
Fix finite constants $D\geq1 $ and $E\geq1$, 
and positive finite constants $\rho\in (0,1)$, $\lambda$, $L$ and $\gamma$.
For each full-rank $d\times p$ matrix $M$, each
$d\times d$ variance/covariance matrix $\Sigma>0$
and each $f_{\tilde{z}}\in {\mathcal F}_{d,20}(D,E)$
there exists a Borel set
${\mathbb U} = {\mathbb U}(M,\Sigma,f_{\tilde{z}}) \subseteq
{\mathcal O}_{d}$ so that
$$
	\sup_{\footnotesize \begin{array}{c}
		M
		\end{array}
	}\;
	\sup_{	
		\Sigma
	}\;
	\sup_{f_{\tilde{z}} \in {\mathcal F}_{d,20}(D,E)}\;
	\nu_{d}({\mathbb U}) 
	\quad
	\stackrel[]{\frac{ p }{\log d} \to 0}{\longrightarrow} 
	\quad
	1
$$
and so that the following holds: If $\Xi_n$ denotes either the quantity
\begin{equation}\label{t1.1}
\sup_{t\in\R} \left|
	\P\Big( \hat{F}_n \leq t \Big) - 
	F_{p, n-p-1, n\Delta}(t) 
	\right|
\end{equation}
or the quantity
\begin{equation}\label{t1.2}
\P\Big(
	\hat{F}_n > F^{-1}_{p, n-p-1, 0}(\alpha)
\Big)
-
\Phi\Big(
	- \Phi^{-1}(\alpha)
	+ \sqrt{n} \Delta \sqrt{ \frac{1-p/n}{2 p/n}}
\Big)
\end{equation}
for some fixed $\alpha\in [0,1]$,
then
$$
\sup_{\footnotesize \begin{array}{c}
		M
	\end{array}
}\;
\sup_{	
	\footnotesize \begin{array}{c}	
		\vartheta, \theta, {\mathcal L}(\epsilon), \mu,
	\Sigma\\ \E|\epsilon/\sigma|^{8+\lambda}\leq L\\
	\Delta< \gamma/\sqrt{n}
\end{array}}\;
\sup_{f_{\tilde{z}}\in {\mathcal F}_{d,20}(D,E)}\;
\sup_{R \in {\mathbb U}}\;
\;\;\Xi_n
	\quad
	\stackrel[\frac{n^2 }{\log d} \to 0, 
		\frac{p}{n}\to \rho]{ n\to\infty}{\longrightarrow} 
	\quad
	0.
$$
This statement continues to hold if the restriction $\Delta<\gamma/\sqrt{n}$
in the last display is replaced by $\Delta < g(n)$ provided that 
$\lim_{n\to\infty} g(n) = 0$.
[Here, the suprema are taken over all full-rank $d\times p$ matrices $M$,
all $\vartheta\in \R$, all $d$-vectors $\theta$ and $\mu$,
all distributions ${\mathcal L}(\epsilon)$ so that $\epsilon$ has
mean zero and finite positive variance, and all symmetric and positive definite
$d\times d$ matrices $\Sigma$, subject to the indicated restrictions.]
\end{theorem}

\begin{remark}\normalfont
\label{pinsker}
Write ${\mathcal R}_N$ and ${\mathcal R}_L$ for the prediction risk
of the Bayes predictor and of the best linear predictor, respectively,
of $y$ given $x$. That is, 
${\mathcal R}_N = \E[ (y - \E[y\|x])^2]$ and
${\mathcal R}_L = \E[ (y - (\alpha+\beta'x))^2]$.
The results of \citet{Ste18a} then entail that,
in the setting of Theorem~\ref{t1}, 
${\mathcal R}_N/{\mathcal R}_L$ converges to one,
uniformly over all the parameters indicated in the last display
of that theorem. In fact, the risk-ratio converges to one uniformly
even if the restriction on $\Delta$ is removed altogether,
and a similar statement holds
for the ratio of conditional risks given $x$, i.e.,
$ \E[ (y - \E[y\|x])^2\|x]/\E[ (y - (\alpha+\beta'x))^2\|x]$.
See Theorem~3.1 of \citet{Ste18a} for a more general form of this
statement under weaker assumptions.
\end{remark}

\begin{remark} \normalfont \label{rateofp}
Although the asymptotic approximations in  Theorem~\ref{t1} require that
$p$ is of the same order as $n$,
we point out that the 
non-central $F$-distribution should still give a reasonable approximation 
to the distribution of the $F$-statistic, i.e., the expression 
in \eqref{t1.1} should be small,  even if $p/n$ is very small, 
and, in particular, if $p$ is fixed while $n$ increases.
This situation is further discussed in \citet[p. 31, Section 3.2.2]{Ste15}
in a setting where $n\to\infty$, $p$ is fixed and $p/\log d\to 0$.
Clearly, the same is not true for the expression in \eqref{t1.2},
because the normal approximation to the $F$ is valid only if 
both degrees of freedom, i.e., $p$ and $n-p-1$, are large. 
The statement  regarding \eqref{t1.2} in Theorem~\ref{t1}
coincides with the conclusion of 
Theorem~1 in \citet{Zho11a} obtained for the correctly specified 
Gaussian error case. Moreover, the Gaussian approximation in \eqref{t1.2}
has the advantage that it is easier to interpret than the more 
complicated distribution function of the non-central 
$F$-distribution in \eqref{t1.1}; see also the discussion 
in \citet[Remark 2.4]{Ste16a}.
\end{remark} 

\section{Simulation analysis} \label{numericalresults} 

Theorem~\ref{t1} is an asymptotic result. In this section, we study
a range of non-asymptotic scenarios through simulation to investigate
how soon these asymptotic approximations become accurate.
We consider a rather small sample size of $n=50$ 
and look at different configurations of the model 
dimensions $d$ and $p$ with $p<d$, and also at different points in
parameter space. 

The theorem contains two asymptotic statements, one about
the distribution of the $F$-statistic and one about the size
of the set $\mathbb U$. For the distribution of the $F$-statistic, 
we compare the rejection probability of the $F$-test under the 
null hypothesis with the nominal significance level $\alpha=0.05$. 
The nominal significance level
provides a natural benchmark.  [Clearly, one can 
also investigate the power of the $F$-test through simulation
experiments, but, unlike the significance level, it is less
obvious what the right benchmark for the power should be.]
In particular, we simulate 1000 independent realizations $F_{j,r}$,
$j=1,\dots,1000$ of the $F$-statistic at sample size $n=50$ 
under the null for each point in parameter space 
(the index $r$ will be explained shortly),
and compare  the empirical significance level 
$\overline{p}_r = 1000^{-1} \sum_{j=1}^{1000} {\mathbf 1}\{ F_{j,r} > 
	F^{-1}_{p,n-p-1,0}(1-\alpha)\}$
with the nominal level $\alpha$.

Gauging the size of $\mathbb U$ is more difficult, because
that set is not given explicitly.  We proceed as follows: We fix all 
the parameters in \eqref{y}--\eqref{z} except for
the orthogonal matrix $R$ in \eqref{z}. We then simulate 100 independent
realizations $R_r$ of $R$, compute $\overline{p}_r$ as outlined
above, $r=1,\dots, 100$, and finally compute 
$\overline{D}=100^{-1} \sum_{r=1}^{100} |\overline{p}_r - \alpha|$.
If $R_r \in \mathbb U$, then $\overline{p}_r$ should be close to $\alpha$,
in view of the last display in Theorem~\ref{t1}.
We use $\overline{D}$ and the empirical distribution of the $\overline{p}_r$, 
$r=1,\dots,100$,
as indicators for the size of $\mathbb U$.

The remaining parameters in \eqref{y}--\eqref{z} and the submodel matrix
$M$ are chosen
as follows for any fixed values of $d$ and $p$:
We do not include an error term in 
the true model, i.e., we set $\sigma^2 = 0$,
because the effect of misspecification becomes more 
pronounced when the error variance $\sigma^2$ is small.\footnote{
	Note that if the error variance $\sigma^2 = \Var[\epsilon_i]$ 
	in the true model $y_i = \theta'z_i + \epsilon_i$ is overly large, 
	i.e., much larger than $\theta' \Sigma\theta$, then
	the scaled true model is essentially given by 
	$y_i/\sigma \approx \epsilon_i/\sigma$. Since the $F$-statistic is 
	scale-invariant and $\epsilon$ is independent of $X$, we then have
	$\hat{F}(X,Y) = \hat{F}(X,Y/\sigma) \approx 
	\hat{F}(X,(\epsilon_i)_{i=1}^n/\sigma) 
	= \hat{F}(X,(\epsilon_i)_{i=1}^n)$. In that case,
	the $F$-statistic will essentially 
	follow the null-distribution and we expect a rejection probability 
	close to the nominal level, irrespective of $\theta$ and $R$.
} 
[Note that the case where $\sigma^2=0$ is not covered by
Theorem~\ref{t1}, but inspection of the proof shows that our
results also apply in this case; cf. Remark~\ref{zerovariance}.]
For $\tilde{z}$, we consider product distributions with zero mean and 
i.i.d. components from the student-$t$ distribution with $2$, $3$ and $5$ 
degrees of freedom, as well as from the centered exponential, uniform, 
Bernoulli$\{-1,1\}$ and Gaussian distributions. [Note that the scaling of 
these distributions is inconsequential, because of the scale-invariance 
of the $F$-statistic $\hat{F}(X,Y)$ in both arguments and the fact that we 
do not include an error term in the full model, i.e., scaling of 
$\tilde{z}_i$ is equivalent to scaling of both 
$y_i = \theta'z_i$ and $x_i = B'z_i$. Similarly, also the scaling of 
$\theta$ and $\Sigma$ has no impact on the value of the $F$-statistic.]
For $\Sigma$, we chose a spiked covariance matrix
$\Sigma = U\text{diag}(\lambda_1,\dots,\lambda_n) U'$ with eigenvalues 
$\lambda_1 = \lambda_2 = 400$ and $\lambda_3=\dots=\lambda_d=1$ and 
an orthogonal matrix of eigenvectors $U$ chosen randomly from the uniform 
distribution on the orthogonal group.\footnote{
	The spiked covariance model 
	corresponds to a factor model where the identity matrix is 
	perturbed by a low rank matrix. It has received much attention 
	in the literature on high dimensional random matrices 
	\citep[e.g.,][]{Bai06a,Cai13a,Don13a,Joh01a}. We have 
	repeated the simulations also with covariance matrices of an 
	AR$(1)$ process and obtained essentially the same results.
} 
The intercept terms $\vartheta$ and $\mu$ are set to zero, for convenience.
For the matrix $M$, which describes the working model, we take
$M$ equal to the $d\times p$ matrix whose $k$-th column is
the $k$-th standard basis vector
in $\mathbb R^d$, $1\leq k \leq p$.
In other words, we consider a sub-model that includes
only the first $p$ regressors (out of $d$).
For the parameter $\theta\in \mathbb R^d$, we need to ensure
that the null hypothesis is satisfied,
i.e., that
$\beta = (M'\Sigma M)^{-1}M'\Sigma\theta = 0$. 
By construction of 
$\Sigma$, $M'\Sigma M$ is regular, and we choose 
$\theta = (I_d-P_{\Sigma M}) V/\|(I_d-P_{\Sigma M}) V\|$, for one realization 
of $V\thicksim N(0,I_d)$, to guarantee that $M'\Sigma\theta = 0$.

The results of the simulations are summarized in Table~\ref{tab:NullMod} 
and Figures~\ref{fig:Boxplots1} and \ref{fig:Boxplots2}. From 
Table~\ref{tab:NullMod}, the overall picture we get is consistent with 
what was predicted by our theory. For all distributions except the Gaussian, 
the average absolute difference between the true (simulated) rejection 
probabilities and the nominal level decreases as $d$ increases.
This phenomenon is most pronounced for the exponential distribution, 
which has a finite moment generating function around the origin, and 
is weakest for the $t(2)$-distribution, which does not even have finite 
variance. For uniformly distributed design, which is bounded, the effect of 
misspecification on the size of the $F$-test is relatively mild already for 
small dimensions. In the Gaussian case, all sub-models of the 
form \eqref{workingmodel} are correct in the
sense that the error $e$ is Gaussian with mean zero and independent of $x$,
so that theoretically the corresponding panel 
in Table~\ref{tab:NullMod} should contain only zeroes. The numbers therefore 
represent only the simulation error and serve as a benchmark for the 
other panels. We also see a monotonic increase, in the deviation of the size 
of the $F$-test from the nominal level, as the dimension $p$ of the 
sub-model increases, which was also suggested by our theory. However, if we 
fix the ratio $p/d=1/2$, i.e., if we move along the staircase pattern in 
each of the panels, except for the heavy tailed distributions $t(3)$ and 
$t(2)$, we still see the effect of misspecification decrease as $d$ 
increases. This suggests that convergence of $n^2 /\log(d)\sim p^2/\log(d)$ 
to zero, as required in Theorem~\ref{t1}, may not be necessary, at least in 
the scenarios considered here.

\begin{table}
\centering
\begin{tabular}{c|cccc | Hcccc} 	
\hline\hline
 $d\backslash p$	& 1		& 2		& 5		& 25		& 	$d\backslash p$	& 1		& 2		& 5		& 25 		\\ \hline
			&		&\multicolumn{2}{c}{$t(5)$}&		&  					&		&\multicolumn{2}{c}{Exp(1)}&	\\
 2				& 0.077	&		&		&		&	2				& 0.141	&		&		&		\\
 4				& 0.056	& 0.076	&		&		&	4				& 0.093	& 0.140	&		&		\\
 10				& 0.032	& 0.047	& 0.066	&		&	10				& 0.052	& 0.071	& 0.109	&		\\
 50				& 0.009	& 0.013	& 0.017	& 0.019	&	50				& 0.014	& 0.015	& 0.020	& 0.033	\\
 100				& 0.007	& 0.008	& 0.009	& 0.010	&	100				& 0.009	& 0.009	& 0.012	& 0.015	\\
 200				& 0.006	& 0.007	& 0.006	& 0.008	&	200				& 0.007	& 0.007	& 0.006	& 0.009	\\ \hline
			&		&\multicolumn{2}{c}{$t(3)$}&		&  					&		&\multicolumn{2}{c}{Unif$[-1,1]$}&	\\
 2				& 0.188	&		&		&		&	2				& 0.025	&		&		&		\\
 4				& 0.158	& 0.225	&		&		&	4				& 0.020	& 0.023	&		&		\\
 10				& 0.122	& 0.167	& 0.238	&		&	10				& 0.011	& 0.014	& 0.016	&		\\
 50				& 0.062	& 0.084	& 0.116	& 0.123	&	50				& 0.006	& 0.006	& 0.007	& 0.007	\\
 100				& 0.048	& 0.061	& 0.081	& 0.082	&	100				& 0.005	& 0.006	& 0.006	& 0.005	\\
 200				& 0.033	& 0.044	& 0.057	& 0.055	&	200				& 0.005	& 0.005	& 0.005	& 0.006	\\ \hline
			&		&\multicolumn{2}{c}{$t(2)$}&		&  					&		&\multicolumn{2}{c}{Gauss}&	\\
 2				& 0.335	&		&		&		&	2				& 0.005	&		&		&		\\
 4				& 0.332	& 0.458	&		&		&	4				& 0.006	& 0.005	&		&		\\
 10				& 0.301	& 0.411	& 0.563	&		&	10				& 0.005	& 0.005	& 0.006	& 		\\
 50				& 0.250	& 0.335	& 0.456	& 0.518	&	50				& 0.005	& 0.006	& 0.005	& 0.005	\\
 100				& 0.228	& 0.314	& 0.412	& 0.457	&	100				& 0.005	& 0.005	& 0.006	& 0.005	\\
 200				& 0.212	& 0.286	& 0.383	& 0.407	&	200				& 0.005	& 0.005	& 0.006	& 0.006	\\ \hline
 \hline
\end{tabular}
\caption{
	Average absolute differences 
	$\bar{D} = \frac{1}{100} \sum_{r=1}^{100} |\bar{p}_r - \alpha|$ of 
	simulated rejection probabilities $\bar{p}_r =\frac{1}{1000} 
	\sum_{j=1}^{1000} \mathbf{1}\{F_{j,r}>F^{-1}_{p, n-p-1,0}(1-\alpha)\}$
	and nominal 
	significance level $\alpha=0.05$ of the $F$-test for $H_0:\beta=0$.
}
\label{tab:NullMod}
\end{table}

\begin{figure}
\includegraphics[width=\textwidth]{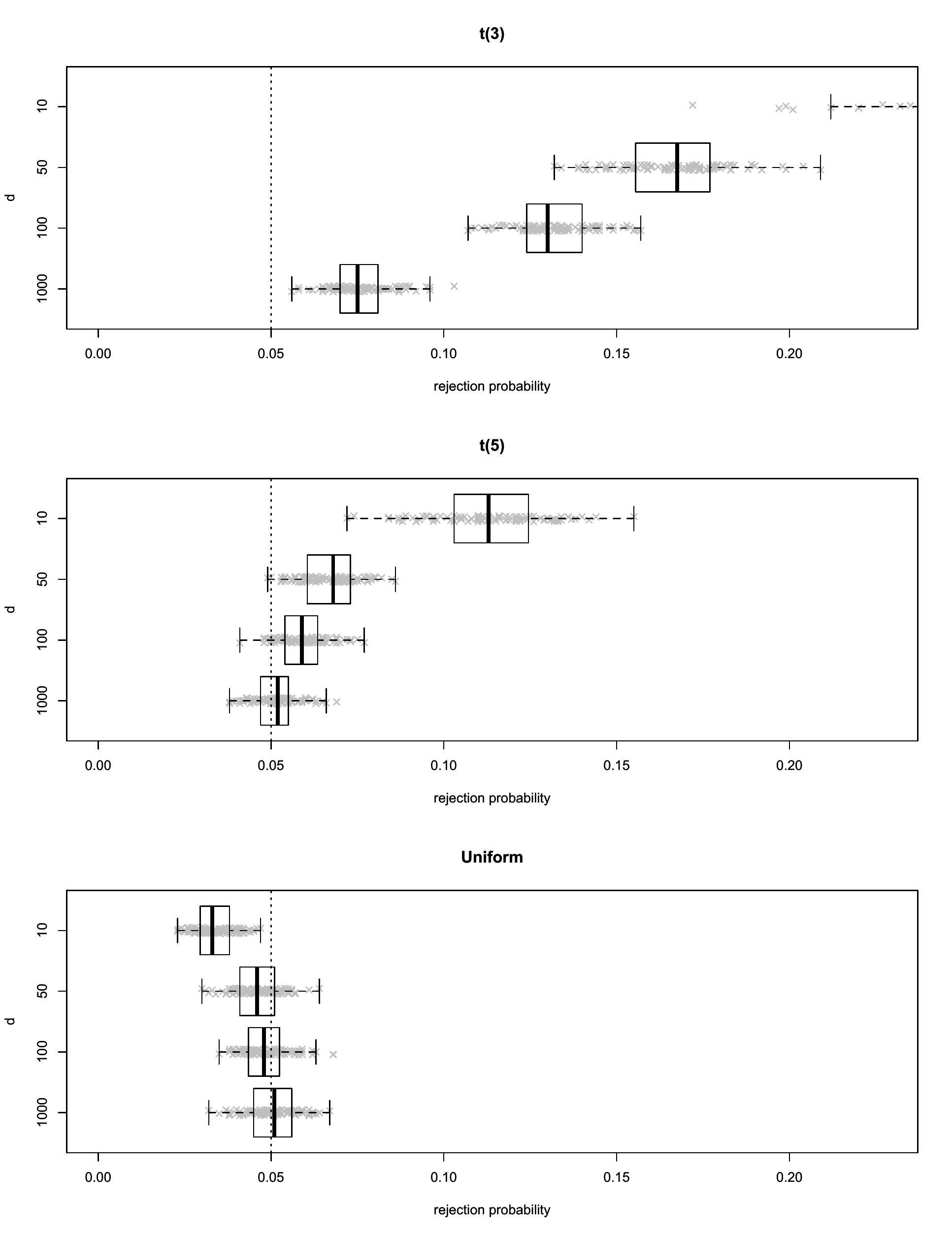} 
\caption{
	Box-plots of simulated rejection probabilities 
	$(\bar{p}_r)_{r=1}^{100}$ (gray crosses) of the $F$-test with 
	$n=50$, $p=5$ and $d=10,50,100,1000$, for different design 
	distributions. Every $r\in\{1,\dots, 100\}$ corresponds to a 
	different $R_r$ applied to the standardized design $\tilde{z}$.
}
\label{fig:Boxplots1}
\end{figure}

\begin{figure}
\includegraphics[width=\textwidth]{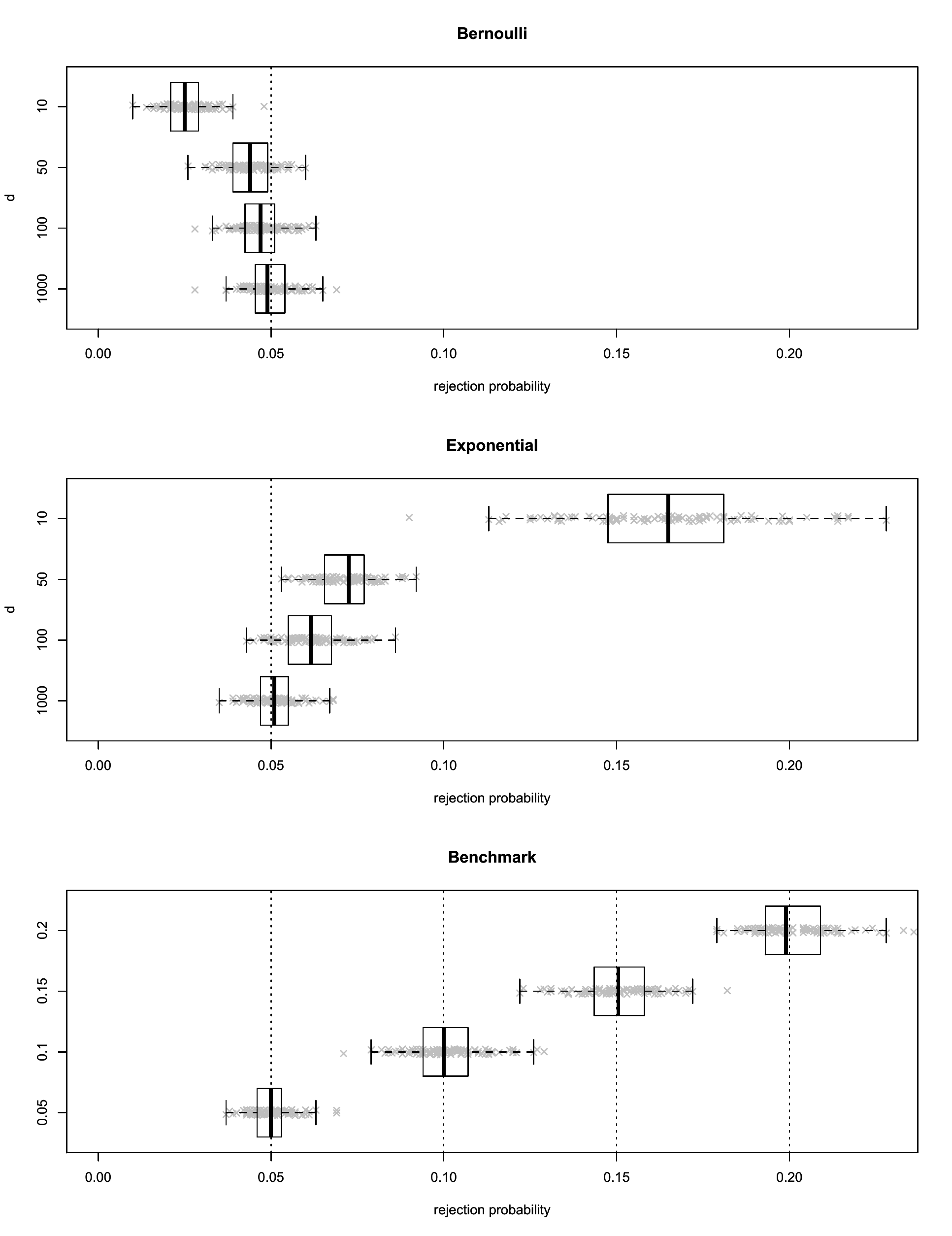} 
\caption{
	Box-plots of simulated rejection probabilities 
	$(\bar{p}_r)_{r=1}^{100}$ (gray crosses) of the $F$-test with $n=50$, 
	$p=5$ and $d=10,50,100,1000$, for Bernoulli$\{-1,1\}$ and exponential 
	design distributions and a benchmark panel of Binomial samples 
	with different success probabilities.
}
\label{fig:Boxplots2}
\end{figure}

In Table~\ref{tab:NullMod}, the effect of the orthogonal matrix $R$
on the actual significance level of the $F$-test was
compressed into one summary statistic, namely the mean absolute deviation 
from the nominal significance level. To get a more comprehensive picture, 
Figures~\ref{fig:Boxplots1} and \ref{fig:Boxplots2} show plots of the 
sample $(\bar{p}_r)_{r=1}^{100}$ (gray crosses) and superimposed box-plots 
for different design distributions. 
Due to limited space we present only the results for sub-models of 
dimension $p=5$. 
In view of Theorem~\ref{t1}, we expect that the size $\mathbb U$, i.e.,
the family of matrices $R$ for which \eqref{t1.1} and \eqref{t1.2}
get small, grows with $d$.
Consequently, we expect that many of the 
$\bar{p}_r$ should be close to $\alpha=0.05$. On the other hand, 
if $d$ is not large then many matrices $R$ will lead to a biased rejection 
probability due to misspecification of the working model. This is exactly 
what we observe in Figures~\ref{fig:Boxplots1} and \ref{fig:Boxplots2}. For 
small values of $d$, the rejection probabilities $\bar{p}_r$ are 
systematically biased and we see some variability of their values due to 
the variation in the choice of $R_r$ (compare benchmark panel in 
Figure~\ref{fig:Boxplots2}). Both the bias and the variability in $\bar{p}_r$ 
reduce when $d$ increases, which is what we expected, as for large $d$, most 
$R_r$ will be favorable and we obtain small misspecification errors 
uniformly over these favorable $R_r$. What is remarkable is the 
systematic over-rejection in case of the $t$- and exponential distribution and 
the under-rejection for Bernoulli and uniformly distributed designs. We 
currently can 
not explain the mechanism that is responsible for this pattern. 
Finally, the benchmark panel shows i.i.d. samples 
$(\tilde{p}_r)_{r=1}^{100}$ with 
$\tilde{p}_r \thicksim \text{Binomial}(1000,\alpha)/1000$ and success 
probabilities $\alpha = 0.05, 0.1, 0.15,0.2$. This provides some idea what 
portion of the variability observed in the other panels is due to random 
simulation error. Clearly, the results in the benchmark panel could have 
been equivalently obtained by repeating the previous simulation for the 
$F$-test with Gaussian design at significance levels 
$\alpha= 0.05,0.1, 0.15, 0.2$.

\section*{Acknowledgments}

The first author's research was  partially supported by 
FWF projects P 26354-N26 and P 28233-N32.

\begin{appendix}

\section{Proofs}

We begin with some preliminary considerations that connect this
paper with the results of \citet{Ste18b}. In particular, we use
Theorem~2.1, parts (ii) and (iii), in that reference
with $Z=\tilde{z}$ and $\tau=1/2$:
If $f_{\tilde{z}}\in\mathcal F_{d,20}(D,E)$, 
then the assumptions of that result are satisfied
in view of Example 3.1 in \citet{Ste18b}.
The theorem guarantees existence of a Borel 
subset $\mathbb G = \mathbb G(f_{\tilde{z}})\subseteq \mathcal V_{d,p}$ 
of the Stiefel manifold $\mathcal V_{d,p}$ of order $d\times p$, that 
depends on the density $f_{\tilde{z}}$, such that for all $t>0$
both
$$
\sup_{B\in\mathbb G} \P\left(\big\| \E[\tilde{z}\| 
	B'\tilde{z}] - BB'\tilde{z}\big\| > t \right) 
$$
and
$$
\sup_{B\in\mathbb G} \P\left(
	\big\|\E[\tilde{z}\tilde{z}'\| B'\tilde{z}] -
		(I_d-B B' + B B' \tilde{z}\tilde{z}' B B')\big\|
	> t \right)
$$
are bounded from above by
\begin{align}\label{approxMeanVar}
	\frac{1}{t} d^{-1/20} + 4\gamma\frac{p}{\log{d}},\\
\end{align}
such that
\begin{align}\label{eq:G1size}
\nu_{d,p}(\mathbb G^c) \;\le\; \kappa d^{-(1-20\gamma\frac{p}{\log{d}})/20},
\end{align}
where $\nu_{d,p}$ denotes the uniform distribution on the Stiefel manifold,
and such that the set $\mathbb G$ is right-invariant under the action of 
$\mathcal O_p$, i.e., $\mathbb G R = \mathbb G$ whenever $R \in \mathcal O_d$.
Here, the constant $\gamma = \gamma(D)$ depends only on $D$, and the 
constant $\kappa = \kappa(E)$ depends only on $E$. 

For any full rank $d\times p$ matrix $M$, any symmetric 
positive definite $d\times d$ matrix $\Sigma$ and 
$f_{\tilde{z}}\in\mathcal F_{d,20}(D,E)$, we define the set 
$$
\mathbb U \;:=\; \mathbb U(M,\Sigma, f_{\tilde{z}}) \;:=\; 
\left\{ R\in \mathcal O_d : R'\Sigma^{1/2} M (M'\Sigma M)^{-1/2} 
	\in \mathbb G(f_{\tilde{z}})\right\}.
$$ 
Now take a random matrix $U$ that is uniformly distributed on 
$\mathcal O_d$ and another random matrix $V$ that is uniformly distributed 
on $\mathcal O_p$, such that $U$ and $V$ are independent, and note that by 
right-invariance of $\mathbb G$,
\begin{align*}
\nu_d(\mathbb U) \; &= \; \P(U \Sigma^{1/2}M(M'\Sigma M)^{-1/2} \in \mathbb G)\\
\;&=\; \P(U \Sigma^{1/2}M(M'\Sigma M)^{-1/2} V \in \mathbb G)
\;=\; \nu_{d,p}(\mathbb G),
\end{align*}
because $\Sigma^{1/2}M(M'\Sigma M)^{-1/2}\in\mathcal V_{d,p}$ 
and $\nu_{d,p}$ is characterized by left and right invariance under the 
appropriate orthogonal groups. 
It follows that $\nu_d(\mathbb U^c)$ is bounded by the expression on 
the right-hand side of \eqref{eq:G1size} whenever 
$f_{\tilde{z}} \in \mathcal F_{d,20}(D,E)$, which establishes the
first claim of Theorem~\ref{t1}. The proof of the second claim
is more elaborate.

The results in the preceding paragraph also show that
the error $e$ in the working model \eqref{workingmodel} 
is such that $\E[e\|x]$ is approximately zero and
$\Var[e\|x]$ is approximately constant, provided that $R \in \mathbb U$:
We first re-write the error $e$ in
a convenient form. Set $\tilde{\theta} = R'\Sigma^{1/2} \theta$
and $\tilde{M} = R' \Sigma^{1/2} M$. Then it is easy to see that
$e =\tilde{\theta}'(I_d-P_{\tilde{M}}) \tilde{z} + \epsilon$
and hence
\begin{align} \label{tmp1}
\begin{split}
\E[e\|x] &\quad=  
	\quad\tilde{\theta}'(I_d-P_{\tilde{M}}) 
	\Big\{\E[\tilde{z}\|P_{\tilde{M}}\tilde{z}]  -
		P_{\tilde{M}}\tilde{z}\Big\} \quad\text{and}
	\\
\E[e^2\|x] - s^2 &\quad=\quad\\
&\hspace{-1cm}
	\tilde{\theta}'(I_d-P_{\tilde{M}}) \Big\{
	\E[\tilde{z}\tilde{z}'\|P_{\tilde{M}}\tilde{z}]
	- ((I_d - P_{\tilde{M}})+P_{\tilde{M}}\tilde{z}\tilde{z}'P_{\tilde{M}})
	\Big\}
	(I_d-P_{\tilde{M}}) \tilde{\theta};
\end{split}
\end{align}
see also \eqref{workingmodel}--\eqref{betas2}.
Our goal is to show that the expressions in the preceding two displays
are approximately zero. To this end, we focus on the expressions
in curly brackets and use Cauchy-Schwarz: For each $t>0$, we have
\begin{align*}
\P( |\E[e\|x]| > t) &\quad\leq\quad 
	\P\left(\Big\|\E[\tilde{z}\|P_{\tilde{M}}\tilde{z}] - 
	P_{\tilde{M}}\tilde{z}\Big\|
	> t / \|(I_d-P_{\tilde{M}})\tilde{\theta}\|\right) \quad\text{and}\\
\P( |\E[e^2\|x] - s^2| > t) &\quad\leq\quad 
\\ & \hspace{-1cm}
	P\left( \Big\| \E[\tilde{z}\tilde{z}'\|P_{\tilde{M}}\tilde{z}] -
		((I_d - P_{\tilde{M}}) + 
			P_{\tilde{M}}\tilde{z}\tilde{z}'P_{\tilde{M}})
	\Big\| > t/\|(I_d-P_{\tilde{M}})\tilde{\theta}\|^2\right).
\end{align*}
Now if $R \in \mathbb U(M,\Sigma,f_{\tilde{z}})$, then it is easy to see that
$\tilde{M}(\tilde{M}'\tilde{M})^{-1/2} \in \mathbb G(f_{\tilde{z}})$.
Because conditioning on $P_{\tilde{M}}\tilde{z}$ is equivalent to
conditioning on $(\tilde{M}'\tilde{M})^{-1/2} \tilde{M}'\tilde{z}$, 
it follows that $\P( |\E[e\|x]| > t)$
is bounded from above by \eqref{approxMeanVar} with
$t$ replaced by $t / \|(I_d-P_{\tilde{M}})\tilde{\theta}\|$
and that $\P( |\E[e^2\|x] - s^2| > t)$
is bounded by \eqref{approxMeanVar}
with $t$ replaced by $t/\|(I_d-P_{\tilde{M}})\tilde{\theta}\|^2$.

The consideration in the preceding paragraph suggests that
the effect of misspecification in \eqref{workingmodel}, where
$\E[e\|x]$ may be non-zero and $\Var[e\|x]$ may be non-constant,
may be negligible in an asymptotic setting where $p/\log d$ becomes small,
provided that $f_{\tilde{z}} \in \mathcal F_{d,20}$ and that
$R \in \mathbb U(M,\Sigma,f_{\tilde{z}})$. This idea is formalized
in  the following two results, which show that the distribution of
certain statistics is unaffected asymptotically if the error $e$
is replaced by a substitute error $e^\ast$ that has mean zero 
and constant variance conditional on $x$.
The following results are stated for sequences 
where the data-generating model \eqref{y}-\eqref{z} and the working model 
\eqref{workingmodel} are allowed to depend on $n$, that is,  a
`triangular array' setting where all parameters depend on $n$.

\begin{lemma}\label{lemma:LinQuad}
Fix finite positive constants $D$ and $E$. For 
every $n\in\N$, let $p_n\le d_n$ be positive integers
so that $n p_n / \log d_n \to 0$ as $n\to\infty$.
For each $n$, consider $(y,z,x)$ as in \eqref{y}--\eqref{x}
but with $d_n$ and $p_n$ replacing $d$ and $p$, respectively,
with $f_{\tilde{z}} \in \mathcal F_{d_n, 20}(D,E)$
and with $R \in \mathbb U(M,\Sigma,f_{\tilde{z}})$.
And for each $n$, consider a sample of $n$ i.i.d. observations 
$(y_i, z_i, x_i)$, $1\leq i \leq n$, of $(y,z,x)$, 
stack the
values of the individual variables into a vector $Y$ and
matrices $Z$ and $X$, respectively, and
write $E = Y - \alpha \iota - X \beta = (e_1,\dots, e_n)'$ 
for the vector of errors from \eqref{workingmodel}. Finally, 
define a vector $E^\ast = (e^\ast_1,\dots,e^\ast_n)'$ of substitute errors
through $e_i^\ast = s(\Var[e_i\|x_i])^{-1/2}(e_i-\E[e_i\|x_i])$.
Then, for every $k\in\R$ and 
(possibly random) symmetric idempotent $n\times n$ matrices $P_n$,
\begin{align}
n^k \|E - E^*\|/s \;&\stackrel{p}{\longrightarrow} \; 0 
	\quad\text{and} \label{eq:LinXi}\\
n^k |E'P_nE - {E^*}'P_nE^*|/s^2\;
	&\stackrel{p}{\longrightarrow}\;0,\label{eq:QuadXi}
\end{align}
as $n\to\infty$.
As a by product, we also obtain that
\begin{align*}
&\max_{i=1,\dots, n} |\Var[e_i\|x_i]/s^2 - 1| \;
	\stackrel{p}{\longrightarrow}\;0.
\end{align*}
\end{lemma}

\begin{proof}
First, note that $\Var[e_i\|x_i] = 
	\Var[y_i\|x_i] = \Var[\theta'z_i\|x_i] + \sigma^2 > 0$, 
so that $e_i^*$ is well defined (almost surely).
For the claim in \eqref{eq:LinXi}, fix $k\in\R$ and $t>0$, and 
consider $\P(n^k\|E-E^*\|/s >t) \le 
	n \P(n^{2k+1} |e_1-e_1^*|^2/s^2>t^2)$.
Now, using the simple observation 
$|\sqrt{\Var[e_1\|x_1]}-s| = 
	|\Var[e_1\|x_1]-s^2|/|\sqrt{\Var[e_1\|x_1]}+s| \le 
	|\Var[e_1\|x_1]-s^2|/s$, 
we get
\begin{align*}
	|e_1-e_1^*|/s &= 
		(s^2 \Var[e_1\|x_1])^{-1/2} 
		\left|e_1(\sqrt{\Var[e_1\|x_1]} - s) + 
		s\E[e_1\|x_1]\right| \\
	&\le
	\frac{s}{\sqrt{\Var[e_1\|x_1]}} 
		\left( \frac{|e_1|}{s}\frac{|\Var[e_1\|x_1] - 
		s^2|}{s^2} + \frac{|\E[e_1\|x_1]|}{s}\right),
\end{align*}
and furthermore
\begin{align}
&\P(n^{2k+1} |e_1-e_1^*|^2/s^2>t^2)\nonumber\\
&\le
\P\left(
	n^{k+1/2}\left| \frac{|e_1|}{s}\frac{|\Var[e_1\|x_1] 
	- s^2|}{s^2} + \frac{|\E[e_1\|x_1]|}{s}\right|
	> t/\sqrt{2}
\right)\notag\\
&\quad\quad + \P\left(
	\frac{s^2}{\Var[e_1\|x_1]}	> 2
\right) \notag\\
&\le
\P\left(
	\left|\frac{\Var[e_1\|x_1]}{s^2} - 1\right|> \frac{1}{2}
\right) 
+
\P\left(
	n^{k+1/2} \frac{|e_1|}{s}\frac{|\Var[e_1\|x_1] - s^2|}{s^2}
	> t/2^{3/2}
\right) \notag\\
&\quad\quad+
\P\left(
	n^{k+1/2}\frac{|\E[e_1\|x_1]|}{s}
	> t/2^{3/2}
\right)\notag\\
\begin{split} \label{tmp2}
&\le
\P\left(
	\frac{|\Var[e_1\|x_1]-s^2|}{s^2}> \frac{1}{2}
\right) 
+
\P\left(
	n^{k+3/2} \frac{|\Var[e_1\|x_1] - s^2|}{s^2}
	> t/2^{3/2}
\right)\\
&\quad\quad+\P\left(
	\frac{|e_1|}{s} > n 
\right) 
+
\P\left(
	n^{k+1/2}\frac{|\E[e_1\|x_1]|}{s}
	> t/2^{3/2}
\right).
\end{split}
\end{align}
The claim \eqref{eq:LinXi} will follow 
if each of the four terms in \eqref{tmp2} is of  the order $o(1/n)$.
Because $f_{\tilde{z}}\in\mathcal F_{d_n,20}(D,M)$
and $R \in \mathbb U(M,\Sigma,f_{\tilde{z}})$,
the considerations leading up to Lemma~\ref{lemma:LinQuad}
apply. Also note that
$\|(I_d-P_{\tilde{M}})\tilde{\theta}\|^2 \leq s^2$.
For the last term in \eqref{tmp2}, we obtain,
for every $t>0$, that
\begin{align*}
\P\left(
	n^{k+1/2}\frac{|\E[e_1\|x_1]|}{s} > t
\right)
\le
t^{-1}n^{k+1/2}d_n^{-1/20}  + 4 \gamma  \frac{p_n}{\log{d_n}},
\end{align*}
and the upper bound goes to zero as $o(1/n)$
in view of the assumption that $n p_n / \log d_n \to 0$.
For the second-to-last term in \eqref{tmp2}, we have
$\P(|e_1|/s>n) \leq n^{-2} \E[e_1^2/s^2] = 1/n^2$. For the second term
in \eqref{tmp2}, we proceed like for the last term in \eqref{tmp2}.
In particular, we obtain, for any $t>0$, that
\begin{align}\label{eq:CondVarop1}
&\P\left(
	n^{k+3/2}\frac{|\Var[e_1\|x_1]-s^2|}{s^2} > t
\right)\\
&\quad\le
\P\left(
	n^{k+3/2}\frac{|\E[e_1^2\|x_1]-s^2|}{s^2} > t/2
\right)
+
\P\left(
	n^{k+3/2}\frac{|\E[e_1\|x_1]|^2}{s^2} > t/2
\right)\notag\\
&\quad\le
\frac{2}{t} n^{k+3/2} d^{-1/20} + 
\left(\frac{2}{t} n^{k+3/2}\right)^{1/2} d^{-1/20} + 
8 \gamma \frac{ p_n }{\log d_n}.
\end{align}
Again, this upper bound goes to zero as $o(1/n)$ because $n p_n/\log d_n \to 0$.
Note that the considerations in  the preceding display
also entail that $\P(\max_{i=1,\dots,n} |\Var[e_i\|x_i]/s^2 - 1|>t) \le 
n \P(|\Var[e_1\|x_1]/s^2 - 1|>t) \to 0$.

For the claim in \eqref{eq:QuadXi}, write
\begin{align*}
|E'P_nE - {E^*}'P_nE^*| 
&=
|(E - E^*)'P_nE + {E^*}'P_n(E-E^*) |\\
&\le
\|E - E^*\| \|E\| + \|E - E^*\| \|E^*\|,
\end{align*}
and note that by definition of $e_1^*$ and the variance decomposition formula, 
we have $\E[e_1^*] = \E[\E[e_1^*\|x_1]] = 0$ and 
$\Var[e_1^*] = \E[\Var[e_1^*\|x_1]] + \Var[\E[e_1^*\|x_1]] = s^2$, so that by 
independence $\|E^*\|/s = O_\P(\sqrt{n})$. Premultiplying 
by $n^k/s^2$ in the previous display and applying 
\eqref{eq:LinXi} finishes the proof of the second claim.
\end{proof}

\begin{lemma}\label{Fstat}
Fix $K \in (0,\infty)$ and an integer $l\geq -1$.
Under the assumptions and in the notation of Lemma~\ref{lemma:LinQuad},
assume that $\E[| \epsilon/\sigma|^4] \leq K$ for each $n$,
that $\Delta = \Var(\beta'x)/\Var(e) = O(n^l)$ and
that $\limsup_{n\to\infty} p_n/n < 1$.
Define substitute data $Y^\ast = \iota \alpha + X\beta + E^\ast$.
Then, for every $k\in \mathbb R$, we have
$$
n^k \left( \hat{F}_n(X,Y) - \hat{F}_n(X,Y^\ast)
\right)\quad\stackrel{p}{\longrightarrow}\quad 0
$$
as $n\to\infty$.
\end{lemma}

\begin{proof}
The idea is to use Lemma~\ref{lemma:LinQuad} to approximate $\hat{F}_n(X,Y)$ 
by $\hat{F}_n(X,Y^*)$. In particular, we will show that on some event 
$C_n$ to be defined below, we have
$$
n^k\left|\hat{F}_n(X,Y) - \hat{F}_n(X,Y^*)\right|
\le n^{k+l+1}|\delta_n^{(1)}-1|\hat{F}_n(X,Y^*)/n^{l+1}  + n^k|\delta_n^{(2)}|,
$$
where $\delta_n^{(1)}$ converges 
to one and $\delta_n^{(2)}$ converges to zero, both at an arbitrary polynomial 
rate in $n$, and where $\hat{F}_n(X,Y^*)/n^{l+1} = O_\P(1)$. 
The probability of $C_n$ will be shown to converge to one.
The claim of the lemma follows from this.

Set $U = [\iota, X]$, where $\iota=(1,\dots,1)'\in\R^n$. 
With this, define the event 
$C_n = \{\det{U'U}\ne0, E'(I_n-P_{U})E>0, 
	{E^*}'(I_n-P_{U})E^*>0\}$. 
On $C_n$, by block matrix inversion, we have 
$[0,I_{p_n}](U'U)^{-1}U' = [X'(I_n-P_\iota)X]^{-1}X'(I_n-P_\iota)$. 
Using the abbreviation $V=(I_n-P_\iota)X$, 
we thus see that $\hat{\beta} = \beta + (V'V)^{-1} V' E$ and that
the $F$-statistic $\hat{F}_n(X,Y)$ can be written as
\begin{align*}
\hat{F}_n(X,Y) &= 
\frac{n-p_n-1}{p_n} \frac{\|V \hat{\beta}\|^2}{ \|(I-P_U)Y\|^2}
\;=\;
\frac{n-p_n-1}{p_n}
	\frac{E'P_VE + 2E'V\beta + \beta'V'V\beta}
	{E'(I_n-P_U)E}\\
&=
\frac{{E^*}'(I_n-P_U)E^*}{E'(I_n-P_U)E} \hat{F}_n(X,Y^*) 
\; + \;
	\frac{E'P_VE - {E^*}'P_V{E^*} + 2(E-{E^*})'V\beta}
	{p_n{E}'(I_n-P_U)E/(n-p_n-1)}.
\end{align*}
This establishes a representation 
$\hat{F}_n(X,Y) = \delta_n^{(1)} \hat{F}_n(X,Y^*) + \delta_n^{(2)}$ on $C_n$. 
On the complement of $C_n$, we set $\delta_n^{(1)} = \delta_n^{(2)}=0$, say.
We next show that for every fixed $k\in\R$, $n^k(\delta_n^{(1)}-1) =o_\P(1)$ 
and $n^k\delta_n^{(2)} =o_\P(1)$.

To verify the claimed properties of these quantities, on $C_n$, consider first
\begin{align*}
\delta_n^{(1)} -1 =
\frac{{E^*}'(I_n-P_U)E^*-E'(I_n-P_U)E}{s^2 
	(n-p_n-1)}\frac{s^2(n-p_n-1)}{E'(I_n-P_U)E}.
\end{align*}
Using Lemma~\ref{lemma:LinQuad}, we see that the first fraction in this 
representation multiplied by $n^k$ converges to zero in probability. 
The second fraction obviously equals $s^2 / \hat{s}^2$.
Define $\hat{s}^{*2}$ like $\hat{s}^2$ (see the discussion following
\eqref{hypotheses}) but with $Y^*$ replacing $Y$.
We show that $\hat{s}^2/s^2 = 
	\hat{s}^{*2}/s^2 + (\hat{s}^2 - 
	\hat{s}^{*2})/s^2 \to 1$
in probability.
To see this, first note that the convergence to zero 
of $(\hat{s}^2 - \hat{s}^{*2})/s^2$ follows again from 
Lemma~\ref{lemma:LinQuad}. For the ratio $\hat{s}^{*2}/s^2$, 
convergence to $1$ in probability follows, 
e.g., from Lemma~C.1 in \citet{Ste16a}, upon 
verifying its assumptions. 
To this end, it remains to show that 
$n^{-1} \sum_{i=1}^n \E[ (e^*_i/s)^4\|x_i] = O_\P(1)$.
Using $(a+b)^4\le 2^{3}(a^4+b^4)$, for $a,b\in\R$, we have
\begin{align*}
\frac{1}{n}\sum_{i=1}^n \E[(e_i^*/s)^4\|x_i] 
&\le
\max_{j=1,\dots,n}\left(\frac{s^2}{\Var[e_j\|x_j]}\right)^2 
\frac{1}{n}\sum_{i=1}^n \E[(e_i/\sigma-\E[e_i/s\|x_i])^4\|x_i] \\
&\le
\max_{j=1,\dots,n}\left(\frac{s^2}{\Var[e_j\|x_j]}\right)^2 
2^4 \frac{1}{n}\sum_{i=1}^n \E[(e_i/s)^4\|x_i].
\end{align*}
The maximum in the preceding display converges to one in probability if 
$\min_j \Var[e_j/s\|x_j]$ converges to one in probability, which follows from 
Lemma~\ref{lemma:LinQuad}. The arithmetic mean of the conditional 
fourth moments is $O_\P(1)$ if the unconditional mean of
forth moments is bounded in $n$. To this end,
note that we have 
$e = \tilde{\theta}'(I_d - P_{\tilde{M}}) \tilde{z}+\epsilon$
and
$s^2 = \|(I_d-P_{\tilde{B}})\tilde{\theta}\|^2 +\sigma^2$; 
cf. \eqref{betas2} and the discussion right before \eqref{tmp1}.
With this, we get
\begin{align*}
(e_i/s)^4 
&= \left( \theta'(I_d-P_{\tilde{M}})\tilde{z}_i /s + \epsilon_i/s\right)^4
\le 
2^3[(\tilde{\theta}'(I_d-P_{\tilde{B}})\tilde{z}_i/s)^4 + 
	(\epsilon_i/s)^4]\\
&\le
2^3[(\tilde{\theta}'(I_d-P_{\tilde{B}})\tilde{z}_i/\|
	\tilde{\theta}'(I_d-P_{\tilde{B}})\|)^4 + (\epsilon_i/\sigma)^4],
\end{align*}
and take expectations. The claim follows now from 
$\E[(\epsilon_i/\sigma)^4]\le K$ and the fact that the fourth 
spherical moment of $\tilde{z}_i$ is uniformly 
bounded in view of Rosenthal's inequality \citep[Theorem 3]{Ros70} 
and the assumption that $f_{\tilde{z}} \in \mathcal F_{d_n,20}(D,E)$. 
Note that this also entails 
$\P(C_n^c) \le \P(\hat{s}^{*2}=0)+\P({\hat{s}_n}^2=0) 
	\le \P(|\hat{s}^{*2}/s^2 - 1|>1/2)+
	\P(|{\hat{s}}^2/s^2 - 1|>1/2) \to 0$.

To see that also $\delta_n^{(2)}$ behaves as desired, first note that on $C_n$, 
\begin{align*}
n^k \delta_n^{(2)} =
\frac{n^k}{p_n}
\left(
	\frac{E'P_V E - {E^*}'P_V{E^*} }{s^2} + \frac{2(E-{E^*})'V\beta}{s^2}
\right)
\frac{s^2}{\hat{s}^2}.
\end{align*}
The factor $n^k/p_n$  can be bounded by $\kappa n^{k-1}$ for some 
constant $\kappa$ by assumption;
the ratio $s^2/\hat{s}^2$ was shown to converge to one in probability in 
the preceding paragraph. The difference of quadratic forms 
converges to zero in probability by Lemma~\ref{lemma:LinQuad}, 
even when multiplied by $\kappa n^{k-1}$. 
Noting that 
$\|V \beta\| = \|(I_n-P_\iota) X \beta \| \leq
	\|(I_n - P_\iota) X (\tilde{M}'\tilde{M})^{-1/2}\| 
	\| (\tilde{M}'\tilde{M})^{1/2} \beta\|$,
the scaled second term in parentheses, i.e.,
$(n^k/p_n) 2 (E-E^*)' V \beta/s^2$, 
can be bounded by
\begin{align*}
2 \kappa n^{k+l/2} \frac{\|E-{E^*}\|}{s} 
	\frac{\|(\tilde{M}'\tilde{M})^{1/2}\beta\|}{s n^{l/2}} 
	\left\|(I_n-P_\iota)X(\tilde{M}'\tilde{M})^{-1/2}\right\|/n, 
\end{align*}
where $n^{k+l/2}\|E-E^*\|/s$ converges to zero in 
probability by Lemma~\ref{lemma:LinQuad} and 
$n^{-l}\beta'(\tilde{M}'\tilde{M})\beta/s^2 = n^{-l} \Delta = O(1)$ 
by assumption. 
It remains to show that the largest singular value of 
$(I_n-P_\iota)X(\tilde{M}'\tilde{M})^{-1/2}/n$ is bounded in probability. 
Due to the projection onto the orthogonal complement of $\iota$, 
the distribution of this quantity does not depend on the parameter 
$\mu$, which is why we may assume that $\mu=0$ for this part of 
the argument. Abbreviate $\bar{X} = X (\tilde{M}'\tilde{M})^{-1/2}$, 
$\bar{x}_i = (\tilde{M}'\tilde{M})^{-1/2}x_i$ and consider 
$\|(I_n-P_\iota)\bar{X}/n\|^2 \le \trace(\bar{X}'\bar{X}/n^2) = 
	\sum_{i=1}^n \|\bar{x}_i\|^2/n^2$. 
Taking expectation, noting that $\E[\|\bar{x}_1\|^2] = p_n$ and 
$p_n/n=O(1)$, we arrive at the desired boundedness in probability.

It remains to show that $\hat{F}_n(X,Y^*)/n^{l+1}=O_\P(1)$. 
To this end, recall that
$\hat{s}^{*2}/s^2 \to 1$ in probability, and one 
easily verifies that
\begin{align*}
\E\left[\frac{ \hat{s}^{*2} }{s^2}\hat{F}_n(X,Y^*)/n^{l+1} \right]
&= 
\E\left[
({E^*}'P_V{E^*} + 2{E^*}'V\beta + \beta'V'V\beta)/(p_ns^2n^{l+1})
\right]\\
&=
\frac{1}{n^{l+1}} + \frac{n-1}{np_n} \frac{\Delta}{n^l} 
	= O(1);
\end{align*}
here, the first equality is obtained by arguing as in the first paragraph
of the proof but with $Y^\ast$ replacing $Y$, and 
the second equality follows upon noting that
$\beta'V'V\beta = \trace(I_n-P_\iota) X\beta \beta'X'$
and that $X\beta$ is a vector with i.i.d. components,
each of which has variance $\beta'M'\Sigma M\beta = s^2 \Delta$.
\end{proof}

\begin{proof}[Proof of Theorem~\ref{t1}]
Define $\mathbb U = \mathbb U(M,\Sigma,f_{\tilde{z}})$
as in the beginning of the appendix and note that 
the first statement in the theorem, concerning $\nu_d(\mathbb U)$,
has already been established there.
For the second statement, concerning $\Xi_n$, let
$p_n\leq d_n$ be positive integers so that
$n^2 p_n/\log d_n \to 0$  and so that $p_n/n \to \rho \in (0,1)$
as $n\to\infty$.
For each $n$, consider a sample of i.i.d. observations $(y_i, z_i, x_i)$,
$1\leq i \leq n$, as in Lemma~\ref{lemma:LinQuad}, so that the underlying
quantities
(i.e., $M$, $\vartheta$, $\theta$, $\mathcal L(\epsilon)$,
$\mu$, $\Sigma$, $\Delta$, $f_{\tilde{z}}$, and $R$) 
satisfy the restrictions in the suprema in the last display
of Theorem~\ref{t1}. 
For given $M$, we stress that the restriction on $\Delta$ implicitly 
also restricts the parameters $\theta$, $\Sigma$ and $\sigma^2$;
see the definition of $\Delta$ at the beginning of Section~\ref{mainresult}
as well as the relations in \eqref{betas2}.
We have to show that $\Xi_n \to 0$ as $n\to\infty$.

Set $a_n = 2(1/p_n+1/(n-p_n-1))$ and 
$b_n = \sqrt{\frac{(1-(p_n+1)/n)(1-1/n)}{2p_n/n}}$ for each $n$,
and define $Y^\ast$ for each $n$ as in Lemma~\ref{Fstat}.
We first show that
\begin{align}\label{eq:FstarConv}
a_n^{-1/2}(\hat{F}_n(X,Y^*) - 1) - \sqrt{n}\Delta b_n 
\quad\xrightarrow[n\to\infty]{w} \quad N(0,1)
\end{align}
by verifying the assumptions of Theorem~2.1(i) in \citet{Ste16a} 
for the sample $(y_i^*,x_i)_{i=1}^n$,
with the symbols $s_n$, $\Delta_\gamma$ and $R_0$ in that reference
equal to $a_n$, $\Delta$, and $[0,I_{p_n}]$, respectively.
In particular, we need to verify conditions (A1).(a,b,c,d) and (A2)
in that reference.
The design conditions (A1).(a,c,d)
are easily verified by use of Lemma~A.2(i) in \citet{Ste16a}.
And our assumptions that $f_{\tilde{z}}\in\mathcal F_{d_n,20}(D,M)$ and 
that $p_n<n-1$ imply condition~(A1).(b).
Assumption (A2) on the scaled errors $e_i^*/s$ is established by 
an argument similar to the one also used in the third paragraph of 
the proof of Lemma~\ref{Fstat} but for the 
$(8+\kappa)$-th moment instead of the fourth moment:
Simply decompose $e_i^* = e^\circ_i \tilde{\eps}_i$, with 
$e^\circ_i = \sqrt{s^2/\Var[e_i\|x_i]}$ and 
$\tilde{\eps}_i = e_i-\E[e_i\|x_i]$, and use 
Lemma~\ref{lemma:LinQuad} as before to get $\max_{i=1,\dots,n}e^\circ_i \to 1$
in probability. Then, the assumption that 
$\E[|\epsilon/\sigma|^{8+\kappa}]\le K$ and the fact that the marginals 
of $\tilde{z}\in \mathcal F_{d_n,20}(D,M)$ have bounded 20th moment, 
together with Rosenthal's inequality establish the boundedness of 
$\E[|\tilde{\eps}_i/s|^{8+\kappa}]$, which is sufficient for (A2). 
Using Lemma~\ref{Fstat} and noting that $a_n^{-1/2} = n^k(1+o(1))$ for some
$k\in \mathbb R$, it follows that \eqref{eq:FstarConv} continues
to hold with $\hat{F}_n(X,Y)$ replacing $\hat{F}_n(X,Y^*)$.

Now standard arguments conclude the proof:
First, note that an appropriately scaled and centered $F$-distributed 
random variable $\mathcal F_{p_n,n-p_n-1,n\Delta}$ with $p_n$ and $n-p_n-1$ 
degrees of freedom and non-centrality parameter $n\Delta$ is 
also asymptotically normal, i.e., 
\begin{align}\label{eq:Fconv}
	a_n^{-1/2}(\mathcal F_{p_n,n-p_n-1,n\Delta} - 1) - 
	\sqrt{n}\Delta b_n\;\xrightarrow[n\to\infty]{w}\; N(0,1),
\end{align}
because $p_n/n\to \rho\in(0,1)$ implies that $p_n\to\infty$.
Hence, we have
\begin{align*}
&\sup_{t\in\R} \left|\P\left(\hat{F}_n(X,Y) \le t\right)
-
\P(\mathcal F_{p_n,n-p_n-1,n\Delta}\leq t) \right|\\
&\quad=
\sup_{t\in\R} \left|
	\P\left(a_n^{-1/2}(\hat{F}_n(X,Y)-1) -\sqrt{n}\Delta b_n 
	\le t \right)\right.\\
&\quad\quad\quad\quad\quad-\left.
	\P\left(a_n^{-1/2}(\mathcal F_{p_n,n-p_n-1,n\Delta}-1) -
	\sqrt{n}\Delta b_n \le t \right)
\right|\\
&\quad\leq
\sup_{t\in\R} \left|
	\P\left(a_n^{-1/2}(\hat{F}_n(Y,X)-1) -
	\sqrt{n}\Delta_\beta b_n \le t \right) - \Phi(t)\right|\\
&\quad\quad\quad+\sup_{t\in\R}
\left|
	\P\left(a_n^{-1/2}(\mathcal F_{p_n,n-p_n-1,n\Delta}-1) -
	\sqrt{n}\Delta_\beta b_n \le t \right) - \Phi(t)
\right|,
\end{align*}
and the last two suprema 
converge to zero in view of Polya's theorem, which establishes 
the $\Xi_n \to 0$ in case $\Xi_n$ equals \eqref{t1.1}.
Finally, it is elementary to verify 
that $\Xi_n$ also converges to zero in case $\Xi_n$ equals 
\eqref{t1.2}: This follows from \eqref{eq:Fconv} with $\hat{F}_n(X,Y)$
replacing $\hat{F}_n(X,Y^*)$, because
the quantiles of the central $F$-distribution satisfy 
$a_n^{-1/2}(F^{-1}_{p_n,n-p_n-1,0}(\alpha)) \to \Phi^{-1}(\alpha)$.
\end{proof}

\begin{remark}\label{zerovariance}\normalfont
Inspection of the proof reveals that the assumption that $\sigma^2$ is positive
is used only to guarantee that  $\Var[e\|x]>0$ almost surely (and hence
also $s^2 = \Var[e] > 0$).
If this assumption is dropped, we thus see that 
$\Xi_n$ (defined in Theorem~\ref{t1}) converges to zero along sequences
of parameters as used in the proof of Theorem~\ref{t1},
provided that $\Var[\theta'z\|x] > 0$ almost surely for each $n$
(as then $\Var[e\|x] = \Var[y\|x]>0$ a.s.).
\end{remark}

\end{appendix}

{
\bibliographystyle{plainnat}
\bibliography{lit}{}
}

\end{document}